\documentclass[10pt,a4paper,oneside,onecolumn,fleqn]{article}
\usepackage{mathrsfs}
\usepackage{amsfonts}
\usepackage{latexsym}
\usepackage{amsmath,amsthm}
\usepackage{epsf}
\usepackage{epsfig}
\usepackage{graphicx}
\usepackage{float}
\usepackage{indentfirst}
\usepackage{multirow}
\usepackage{bigdelim}
\usepackage{cite}
\usepackage{amsbsy}
\usepackage{amssymb, amsmath}
\usepackage{epsfig}
\usepackage{indentfirst}
\usepackage{color}
\usepackage{subfigure}
\usepackage{caption}

\renewcommand{\arraystretch}{1}

\theoremstyle{plain}
\newtheorem{thm}{\bf Theorem}[section]
\newtheorem{con}[thm]{\bf Construction}

\newtheorem{lem}[thm]{\bf Lemma}

\theoremstyle{definition}
\newtheorem{defi}{\bf Definition}
\newtheorem{exam}{\bf Example}

\theoremstyle{remark}

\title{\bf Constructions of regular sparse anti-magic squares$^\triangle$ }
\author{ Guangzhou Chen$^{1*}$, Wen Li$^2$, Ming Zhong$^3$, Bangying Xin$^2$    \\
\small1. {\it Henan Engineering Laboratory for Big Data Statistical Analysis and Optimal Control, }\\
\small {\it School of Mathematics
and Information Science, Henan Normal University, Xinxiang, 453007, P.R.China}\\
\small2. {\it School of Science, Xichang University, Sichuan, 615000, P.R.China}\\
\small3. {\it Central Primary School of Tingzi town,  Dazhou district, Sichuan, 635011, P.R. China}\\
}
\date{}
\begin{document}
\maketitle
\begin{center}
\begin{minipage}{16cm}
{\bf Abstract}

\vspace{0.2cm} \ \ \ \ \ \ Graph labeling is a well-known and intensively investigated problem in graph
theory. Sparse anti-magic squares are useful in
constructing vertex-magic labeling for graphs. For
positive integers $n,d$ and $d<n$, an
$n\times n$ array $A$  based on $\{0,1,\cdots,nd\}$ is called \emph{a
sparse anti-magic square of order $n$ with density $d$},
denoted by SAMS$(n,d)$, if each element of $\{1,2,\cdots,nd\}$ occurs exactly one entry of $A$,
and its row-sums, column-sums and two main
diagonal sums constitute a set of $2n+2$ consecutive integers. An
SAMS$(n,d)$ is called \emph{regular} if there are exactly $d$ positive
entries in each row, each column and each main diagonal. In this
paper, we investigate the existence of regular sparse anti-magic squares
of order $n\equiv1,5\pmod 6$, and it is proved that for any $n\equiv1,5\pmod 6$, there exists a regular SAMS$(n,d)$
if and only if $2\leq d\leq n-1$.

\end{minipage}
\end{center}

\vskip 0.5cm

 {  \begingroup\makeatletter  \let\@makefnmark\relax  \footnotetext{$^\triangle$ This work was supported by National Natural Science Foundation of China (Grant Nos. 11871417 and 71962030).\\
\mbox{}\hspace{0.18in} $^*$Corresponding author:  G. Chen (chenguangzhou0808@163.com) }

\section{Introduction}

Magic squares and their various generalizations have been objects of
interest for many centuries and in many cultures. A lot of work has
been done on the constructions of magic squares, for more details,
the interested reader may refer to \cite{Abe,Ahmed,Andrews,Hand} and
the references therein.

\emph{An anti-magic square of order $n$} is an $n\times n$ array
with entries consisting of $n^2$ consecutive nonnegative integers
such that the row-sums, column-sums and two main diagonal sums
constitute a set of consecutive integers. Usually, the main diagonal
from upper left to lower right is called \emph{the left diagonal},
another is called \emph{the right diagonal}. The existence of an
anti-magic square has been solved completely by Cormie et al (\cite{Js,J1}).
It was shown that there exists an anti-magic square
of order $n$ if and only if $n\geq 4$.

Sparse magic square had played an important role in the construction of sparse anti-magic square. For
positive integers $n$ and $d$ with $d<n$, an $n\times n$ array $A$ based on
$\{0,1,\cdots,nd\}$ is called a \emph{sparse magic square of order $n$
with density $d$}, denoted by SMS$(n,d)$, if each element of $\{1,2,\cdots,nd\}$ occurs exactly one entry of $A$,
and its row-sums, column-sums and two main diagonal sums is the same. An SMS$(n,d)$ is
called \emph{regular} if there exist exactly $d$ non-zero elements in each
row, each column and each main diagonal. The existence of a regular
SMS$(n,d)$ has been solved completely by Li et al (\cite{Li}). It was
shown that for any positive integers $n$ and $d$ with $d<n$, there exists a regular SMS$(n,d)$ if and only if $d\geq
3$ when $n$ is odd and  even $d\geq 4$ when $n$ is even.

Sparse anti-magic squares are generalizations of anti-magic
squares. For positive integers $n$ and $d$ with $d<n$, let $A$ be an $n\times
n$ array with entries consisting of $0,1,\cdots,nd$ and let $S_A$ be
the set of row-sums, column-sums and two main diagonal sums of $A$.
We call $S_A$ the \emph{sum set} of $A$.
Then $A$ is called a \emph{sparse anti-magic square of order $n$
with density $d$}, denoted by SAMS$(n,d)$, if each element of $\{1,2,\cdots,nd\}$ occurs exactly
one entry of $A$ and $S_A$ consists of
$2n+2$ consecutive integers. In \cite {G4}, an SAMS$(n,d)$ is also
called a \emph{sparse totally anti-magic square}. An SAMS$(n,d)$ is
called \emph{regular} if all of its rows, columns and two main
diagonals contain $d$ positive entries.  As an example, a regular
SAMS$(5,2)$ is listed below.

 \begin{center}
 {\renewcommand\arraystretch{0.8}
\setlength{\arraycolsep}{3.5pt}
\footnotesize
        $A=\begin{array}{|c|c|c|c|c|}      \hline
1&	6&	&&		\\\hline
	&&	2&	8	&\\\hline
5&&&&				3\\\hline
&	9&	7&&	\\\hline
	&&&		4&	10\\\hline
\end{array}$\ .}
\end{center}
\noindent Here, empty entries of $A$ indicate 0.
It is readily checked that the element set of $A$ consists of
$0,1,2,\cdots,10$, $S_A=\{6,7,\cdots,16,17\}$ and all of its
rows, columns and two main diagonals contain $2$ positive entries.

Sparse anti-magic squares and sparse magic squares are useful in graph theory. In paticular,
they can be used to construct the vertex-magic total labeling for
bipartite graphs, trees and cubic graphs, see \cite {G4,GrayJ1,GrayJ2,GrayJ3,Zhux,Zhux1} and the references therein.

Recently, the authors Chen et al(\cite{chen1,chen2,chen3}) proved the following results.
\begin{lem}(\cite{chen1})\label{01}
There exists a regular SAMS$(n,n-1)$ if and only if $n\geq 4$.
\end{lem}

\begin{lem}(\cite{chen2})\label{02}
There exists a regular SAMS$(n,n-2)$ if and only if $n\geq 4$.
\end{lem}

\begin{lem}(\cite{chen3})\label{03}
There exists a regular SAMS$(n,d)$ for $d\in\{3,5\}$ if and only if $n\geq d$.
\end{lem}

In this paper, we investigate the existence of regular sparse anti-magic
squares of order $n\equiv1,5 \pmod 6$ and obtain the
following theorem.
\begin{thm}\label{mainth}
For any $n\geq 5$ and $n\equiv1,5 \pmod 6$, there exists a regular SAMS$(n,d)$ if and only if $2\leq d\leq n-1$.
\end{thm}

For convenience, the following notations are used throughout this paper.
Let $\mathbb{Z}$ be the set of integers, $I_n=\{1,2,3,\cdots, n\}$ and
we always use $I_m$ and $I_n$ to label the rows and columns of an $m\times n$ array respectively.
Let $a,b\in\mathbb{Z}$ and $[a,b]$ be the set of integers $v$
such that $a\leq v\leq b$. Suppose $A$ is an array based on $\mathbb{Z}$, let
$G(A)$, $R(A)$ and $C(A)$ be the set of non-zero elements,
the set of row-sums and the set of column-sums of $A$ respectively.
Let $l(A)$ and $r(A)$ be the sum of the elements in the left diagonal and
the right diagonal of $A$ respectively. Then $S_A=R(A)\cup C(A)\cup \{l(A),r(A)\}$.
Let $a$ and $n$ be integers,
\begin{center}
  $  \langle a\rangle_n= \left\{
          \begin{array}{ll}
              r,     &  if \ \ \ n\nmid a,\ a = mn+r\ \ and\ \ 0<r<n, \\
               n,  & if \ \ \ n|a.\\
          \end{array}
       \right.$
\end{center}
Clearly, $1\leq\langle a\rangle_n\leq n$.

The remainder of this paper is organized as follows. In Section 2, we
show that there exists a regular SAMS$(n,2)$ for $n\equiv1,5 \pmod 6$ and $n\geq 5$ via direct construction.
In Section 3, we introduce a symmetric forward diagonals array which is important
building block in the construction for a regular sparse anti-magic squares.
In Section 4, we prove that there exists a regular SAMS$(n,d)$ for $n\equiv1,5 \pmod 6$ and $d\in [6,n-3]$.
Finally, the proof of Theorem \ref{mainth} is presented in Section 5.

\section{The existence of a regular SAMS$(n,2)$ for $n\equiv1,5 \pmod 6$ and $n\geq5$ }

In this section, we shall prove that there exists a regular SAMS$(n,2)$ for $n\equiv1,5 \pmod 6$ and $n\geq5$.
The idea of our construction is divided into three steps. Firstly, we  give a special array $A$ and a Latin square $B$.
Secondly we shall put the elements of $A$ into the Latin square $B$ to obtain $W$
such that $W$ is a regular SAMS for $n\equiv 1 \pmod 6$ and a near regular SAMS for $n\equiv5 \pmod 6$.
Furthermore, for $n\equiv 5 \pmod 6$, we need to adjust some columns of $W$ to obtain a regular SAMS.

We need the definition of Latin square in the proof of the following.
A \emph{Latin square} of order $n$ is an $n\times n$ array in which each cell contains a
single symbol from an $n$-set $S$, such that each symbol occurs exactly once in each
row and exactly once in each column. A \emph{transversal} in a Latin square of order $n$ is
a set of $n$ cells, one from each row and
column, containing each of the $n$ symbols exactly once.
A Latin square of order $n$ is a \emph{diagonal Latin square} if two main diagonals are transversals.

\begin{thm}\label{SAMS(n,2)}
There exists a regular SAMS$(n,2)$ for $n\equiv1,5 \pmod 6$ and $n\geq5$.
\end{thm}
\begin{proof}
For each $n\equiv1,5 \pmod 6$ and $n\geq5$, it can be written as $n=2m+1$, where $m>1$.
Construct a special $2\times n$ array $A=(a_{i,j})$
over $[1,4m+2]$, where $i=1,2$, $j\in I_n$ and

$$a_{1,j}=\left\{
\begin{array}{lll}
n+j, & j\in[1,m-1],\\
n+j+1, & j\in[m,2m],\\
n, & j=2m+1,\\
\end{array}
\right. \ \ \ \ \
a_{2,j}=\left\{
\begin{array}{lll}
j, & j\in[1,m],\\
3m+1,& j= m+1, \\
j-1, & j\in [m+2,2m+1].\\
\end{array}
\right.$$

Let $R_k$, $k=1,2$, be the set of the elements in the $k$-th row of $A$.
It is easy to see that
\vskip 6pt

\mbox{}\hspace{1.0in}
$R_1=[n+1,n+m-1]\cup[n+m+1,n+2m+1]\cup\{n\}$

\mbox{}\hspace{1.22in} $=[2m+2,3m]\cup[3m+2,4m+2]\cup\{2m+1\}$

\mbox{}\hspace{1.22in} $=[2m+1,4m+2]\setminus\{3m+1\}$.

\begin{center}
$R_2=[1,m]\cup\{3m+1\}\cup[m+1,2m]=[1,2m]\cup\{3m+1\}$.
\end{center}
\noindent
Then we have $R_1\cup R_2=[1,4m+2]$.

Let $S_1$ and $S_2$ be the set of column-sums and forward diagonal-sums respectively.
By a simple calculation, we have

$S_1=\bigcup\limits_{j=1}^{n}\{a_{1,j}+a_{2,j}\}$

\mbox{}\hspace{0.18in}$=\bigcup\limits_{j=1}^{m-1}\{a_{1,j}+a_{2,j}\}\bigcup\{a_{1,m}+a_{2,m},
a_{1,m+1}+a_{2,m+1}\}\bigcup\limits_{j=m+2}^{2m}\{a_{1,j}+a_{2,j}\}\bigcup\{a_{1,2m+1}+a_{2,2m+1}\}$

\mbox{}\hspace{0.16in} $=\bigcup\limits_{j=1}^{m-1}\{n+2j\}\bigcup\{2n,3n+1\}\bigcup\limits_{j=m+2}^{2m}\{n+2j\}\bigcup\{n+2m\}$.

\mbox{}\hspace{0.16in} $=[\bigcup\limits_{j=1}^{2m}\{n+2j\}\setminus\{2n+1\}]\bigcup\{2n,3n+1\}$.

$S_2=\bigcup\limits_{j=1}^{n-1}\{a_{1,j}+a_{2,j+1}\}\bigcup\{a_{1,2m+1}+a_{2,1}\}$

\mbox{}\hspace{0.18in}$=\bigcup\limits_{j=1}^{m-1}\{a_{1,j}+a_{2,j+1}\}
\bigcup\{a_{1,m}+a_{2,m+1}\}\bigcup\limits_{j=m+1}^{2m}\{a_{1,j}+a_{2,j+1}\}\bigcup\{a_{1,2m+1}+a_{2,1}\}$

\mbox{}\hspace{0.16in} $=\bigcup\limits_{j=1}^{m-1}\{n+2j+1\}\bigcup\{(n+m+1)+(3m+1)\}
\bigcup\limits_{j=m+1}^{2m}\{n+2j+1\}\bigcup\{n+1\}$

\mbox{}\hspace{0.16in} $=[\bigcup\limits_{j=1}^{2m}\{n+2j+1\}\setminus\{2n\}]\bigcup\{3n,n+1\}$.
\vskip 6pt
\noindent It follows that $S_1\cup S_2=[n+1,3n+1]\backslash\{2n+1\}$.

Let $B=(b_{i,j})$, where $b_{i,j}=\langle2i+j-1\rangle_n,$ $i,j\in I_n$, note that $n\equiv1,5 \pmod 6$ and $n\geq5$,
then it is easy to check that $B$ is a diagonal Latin square of order $n$ over $I_n$ with
the property

$ b_{n+1-i,n+1-j}=\langle2(n+1-i)+(n+1-j)-1\rangle_n=\langle3n+1-(2i+j-1)\rangle_n=(n+1)-b_{i,j},$ i.e.
$$ b_{i,j}+b_{n+1-i,n+1-j}= n+1.$$
For each $j\in I_n$, define
\begin{center}
$f(x,j)=i$ \emph{if} $b_{i,j}=x$, that is, $f(b_{i,j},j)=i$
\end{center}
and let
\begin{center}
$g(s)=\langle m+2s-1\rangle_n$,\ $s\in I_n$,
\end{center}
then it is easy to see that for each $j\in I_n$, $f(x,j)$ is a bijection function from $I_n$ to $I_n$
since $B$ is a Latin square over $I_n$, $b_{i,j}$ can be regarded as the inverse of $f(x,j)$ for any given $j\in I_n$,
and $g$ is also a bijection function from $I_n$ to $I_n$.

We put $a_{1,s}$ and $a_{2,s}$, $s\in I_n$, into the entries $(f(m,g(s)),g(s))$ and
$(f(m+2,g(s)),g(s))$ of $B$ respectively, the other entries of $B$ are filled by $0$, denoted by $W=(w_{i,j})$,
where $i,j\in I_n$, that is, $w_{f(m,g(s)),g(s)}=a_{1,s}$ and $w_{f(m+2,g(s)),g(s)}=a_{2,s}$.

It is clear that the entries of the $s$-th column, $s\in I_n$, of $A$ are all putted into the
$g(s)$-column of $W$, so the non-zero elements in the same column of $W$
is just in the same column of $A$, then $C(W)=S_1$. We will show that $a_{1,s}$ and $a_{2,\langle s+1\rangle_n}$,
$s\in I_n$, are in the same row of $W$, we  need only to prove that for any $s\in I_n$,
$f(m,g(s))=f(m+2,g(s+1))$. Without loss of generality, suppose that $f(m,g(s))=\xi$, we have
$b_{\xi, g(s)}=\langle2\xi+g(s)-1\rangle_n=m$ by the definition of function $f$, and also have
\begin{center}
$b_{\xi, g(s+1)}=\langle2\xi+g(s+1)-1\rangle_n=\langle2\xi+g(s)+2-1\rangle_n=\langle2\xi+g(s)-1\rangle_n+2=m+2$.
\end{center}
It follows that $f(m+2,g(s+1))=\xi$, then the non-zero elements in the same row of $W$
is just in the forward diagonal of $A$, so $R(W)=S_2$.
It is clear that
\begin{center}
$R(W)\cup C(W)=S_1\cup S_2=[n+1,3n+1]\backslash\{2n+1\}$.
\end{center}

Next,  we shall consider the elements in the two main diagonals of $W$.
There are exactly two non-zero elements in the right diagonal of $W$ according to
the definition of the diagonal Latin square $B$.
It is easy to calculate that
$$a_{1,m+2}=w_{f(m,g(m+2)),g(m+2)}=w_{m,m+2}$$
since $g(m+2)=m+2$,\ \
$b_{i,j}=\langle 2i+j-1\rangle_n=\langle 2m+(m+2)-1\rangle_n=m$ when $i=m,\ j=m+2$, and
$f_{m,j}=i$, i.e. $f_{m,m+2}=m$,
and
$$a_{2,m+1}=w_{f(m+2,g(m+1)),g(m+1)}=w_{m+2,m}$$
since $g(m+1)=m$,\ \
$b_{i,j}=\langle 2i+j-1\rangle_n=\langle 2(m+2)+m-1\rangle_n=m+2$ when $i=m+2,\ j=m$, and
$f_{m+2,j}=i$, i.e. $f_{m+2,m}=m+2$.
Hence the sum of the elements in the right diagonal of $W$ is
$$w_{m,m+2}+w_{m+2,m}=a_{1,m+2}+a_{2,m+1}=(n+m+2+1)+(3m+1)=3n+2.$$

We shall divided it into two cases to deal with the left diagonal-sum below.

\textbf{Case 1:} For $n\equiv 1 \pmod 6$ and $n\geq 7$, it can be written as $n=6k+1$, where $k\geq 1$.
Note that $n=2m+1$, then $m=3k$.

There are exactly two non-zero elements in the left diagonal of $W$ according to
the definition of the diagonal Latin square $B$.
By simple calculation we have
$$a_{1,k+1}=w_{f(m,g(k+1)),g(k+1)}=w_{n-k,n-k}$$
since $g(k+1)=n-k$,\ \
$b_{n-k,n-k}=<2(n-k)+(n-k)-1>_n=3k=m$ and
$f(m,n-k)=n-k$, and
$$a_{2,n+1-k}=w_{f(m+2,g(n+1-k)),g(n+1-k)}=w_{k+1,k+1}$$
since $g(n+1-k)=k+1$,\ \
$b_{k+1,k+1}=<2(k+1)+(k+1)-1>_n=m+2$ and
$f(m+2,k+1)=k+1$.
Then the sum of the elements in the left diagonal of $W$ is
$$w_{k+1,k+1}+w_{n-k,n-k}=a_{2,n+1-k}+a_{1,k+1}=(n+1-k-1)+(n+k+1)=2n+1.$$
So, $W$ is a regular SAMS$(n,2)$.

\textbf{Case 2:}\ \ For $n\equiv 5 \pmod 6$ and $n\geq 5$,
it can be written as $n=6k-1$, where $k\geq 1$, then $m=3k-1$.

When $k=1$, a regular SAMS$(5, 2)$ is given as an example in Section 1.

When $k>1$, it is easy to check that there are also exactly two non-zero elements
in the left diagonal of $W$ according to the definition of the diagonal Latin square $B$, but their
sum is not equivalent to $2n+1$. In fact,
$$a_{1,5k}=w_{f(m,g(5k)),g(5k)}=w_{k,k}$$
since $g(5k)=k$,\ \
$b_{k,k}=<2k+k-1>_n=3k-1=m$ and
$f(m,k)=k$, and
$$a_{2,k+1}=w_{f(m+2,g(k+1)),g(k+1)}=w_{n+1-k,n+1-k}$$
since $g(k+1)=n+1-k$,\ \
$b_{n+1-k,n+1-k}=<2(n+1-k)+(n+1-k)-1>_n=3k+1=m+2$ and
$f(m+2,n+1-k)=n+1-k$.
So the sum of the elements in the left diagonal of $W$ is
$$w_{k,k}+w_{n+1-k,n+1-k}=a_{1,5k}+a_{2,k+1}=(5k+1+n)+(k+1)=2n+3\neq 2n+1.$$

The array $W^*=(w_{i,j}^*)$, $i,j\in I_n$, is obtained by exchanging column $k$ with column
$k+2$ and exchanging column $n+1-k$ with column $n+1-k-2$ of $W$. We list the elements in the
columns $k,k+2,n+1-k,n+1-k-2$ of $B$, $W$ and $W^*$ in the following tables respectively.
 \begin{center}
 {\renewcommand\arraystretch{0.7}
\setlength{\arraycolsep}{2pt}
\footnotesize
        $\begin{array}{|c|c|c|c|}
  \multicolumn{4}{c}{B }\\        \hline
	&	j=k	&	j=k+1	&	j=k+2	\\\hline
i=k-1	&	m-2	&	m-1	&	m	\\\hline
i=k	&	\emph{\textbf{m}}	&	m+1	&	m+2	\\\hline
i=k+1	&	m+2	&	m+3	&	m+4	\\\hline
\end{array}
  \and \hspace{20pt}
\begin{array}{|c|c|c|}
       \multicolumn{3}{c}{W }\\        \hline
	   	j=k	&	j=k+1	&	j=k+2	\\\hline
	&		    &	a_{1,5k+1}	\\\hline
	{\color{red}{a_{1,5k}}}	&	        &	a_{2,5k+1}	\\\hline
	a_{2,5k}&	     	&	\\\hline
\end{array}
  \and \hspace{20pt}
\begin{array}{|c|c|c|}
        \multicolumn{3}{c}{W^* }\\        \hline
	j=k	&	j=k+1	&	j=k+2	\\\hline
a_{1,5k+1}		&		    &		\\\hline
{\color{red}{a_{2,5k+1}}}	&	        &		a_{1,5k}	\\\hline
	&	     	&a_{2,5k}	\\\hline
\end{array}$}
         \end{center}

\begin{center}
 {\renewcommand\arraystretch{0.7}
\setlength{\arraycolsep}{0.8pt}
\footnotesize
        $\begin{array}{|c|c|c|c|}
  \multicolumn{4}{c}{B }\\        \hline
	&	j=n+1-k-2	&	j=n+1-k-1	&	j=n+1-k	\\\hline
i=n+1-k-1	&	m-2	&	m-1	&	m	\\\hline
i=n+1-k	&	\emph{\textbf{m}}	&	m+1	&	m+2	\\\hline
i=n+1-k+1	&	m+2	&	m+3	&	m+4	\\\hline
\end{array}
  \and \hspace{20pt}
\begin{array}{|c|c|c|}
       \multicolumn{3}{c}{W }\\        \hline
	   	j=n+1-k-2	&	j=n+1-k-1	&	j=n+1-k	\\\hline
	&		    &	a_{1,k+1}	\\\hline
	a_{1,k}	&	        &	{\color{red}{a_{2,k+1}}}	\\\hline
	a_{2,k}&	     	&	\\\hline
\end{array}$}
         \end{center}

\begin{center}
 {\renewcommand\arraystretch{0.8}
\setlength{\arraycolsep}{2pt}
\footnotesize
        $\begin{array}{|c|c|c|c|}
        \multicolumn{4}{c}{W^* }\\        \hline
&		j=n+1-k-2	&	j=n+1-k-1	&	j=n+1-k	\\\hline
i=n+1-k-1& a_{1,k+1}		&		    &		\\\hline
i=n+1-k	& a_{2,k+1}	&	        &		{\color{red}{a_{1,k}}}	\\\hline
i=n+1-k+1	&	&	     	&a_{2,k}	\\\hline
\end{array}$}
         \end{center}
Hence $$w_{k,k}^*+w_{n+1-k,n+1-k}^*=w_{k,k+2}+w_{n+1-k,n+1-k-2}=a_{2,5k+1}+a_{1,k}=(5k+1-1)+(n+k)= 2n+1.$$
The set of row-sums, column-sums and the right diagonal-sum of $W^*$
is the same as that of $W$.
Then $W^*$ is a regular SAMS$(n,2)$.
\end{proof}
\noindent\textbf{Remark 1} \ \ For any array $C=(c_{i,j})_{n\times n}$, let $\Omega(C)=\{(i,j)|c_{i,j}\neq0,\ i,j\in I_n\}$ and the notation is used in the rest of the paper.  In the proof of Theorem \ref{SAMS(n,2)},
we have
\begin{center}
$\Omega(W)=\{(i,j)|b_{i,j}\in\{m,m+2\},\ i,j\in I_n\}$,
\end{center}
and
\begin{center}
$\Omega(W^*)\subset\{(i,j)|b_{i,j}\in\{m-2,m,m+2,m+4\},\ i,j\in I_n\}$
in \textbf{Case 2},
\end{center}
 which can be used in the proof later.

\vskip 4pt
To illustrate the proof of Theorem \ref{SAMS(n,2)}, we give an example in the following.
\begin{exam}\label{Ex1}
There exists a regular SAMS$(7,2)$.
\end{exam}
\begin{proof}
By the proof of Theorem \ref{SAMS(n,2)}, take $n=7$, then $m=3$ and $m+2=5$.

\begin{center}
$A=\renewcommand\arraystretch{0.8}\left(
        \begin{smallmatrix}
 8	&	9	&	11	&	12	&	13	&	14	&	7	\\
1	&	2	&	3	&	10	&	4	&	5	&	6	\\
        \end{smallmatrix}
        \right),$\ \ \
  $B=\renewcommand\arraystretch{0.8}\left(
        \begin{smallmatrix}
2	&	\textbf{3}	&	4	&	\textbf{5}	&	6	&	7	&	1	\\
4	&	\textbf{5}	&	6	&	7	&	1	&	2	&	\textbf{3}	\\
6	&	7	&	1	&	2	&	\textbf{3}	&	4	&	\textbf{5}	\\
1	&	2	&	\textbf{3}	&	4	&	\textbf{5}	&	6	&	7	\\
\textbf{3}	&	4	&	\textbf{5}	&	6	&	7	&	1	&	2	\\
\textbf{5}	&	6	&	7	&	1	&	2	&	\textbf{3}	&	4	\\
7	&	1	&	2	&	\textbf{3}	&	4	&	\textbf{5}	&	6	\\
        \end{smallmatrix}
        \right).$
\end{center}
It is readily checked that $g(1)=\langle m+2-1\rangle_n=m+1=4$,  $f(3,4)=7$ and $f(5,4)=1$ since $b_{7,4}=3$
and $b_{1,4}=5$, then $w_{7,4}=a_{1,1}=8$ and $w_{1,4}=a_{2,1}=1$, and so on.
The array $W$ is obtained in the following,
\begin{center}
 {\renewcommand\arraystretch{0.8}
\setlength{\arraycolsep}{3pt}
\footnotesize
\hspace{15pt}
$W=\begin{array}{|c|c|c|c|c|c|c|}\hline
&	7	&		&	1	&		&		&		\\\hline
	&	6	&		&		&		&		&	14	\\\hline
&		&		&		&	13	&		&	5	\\\hline
	&		&	12	&		&	4	&		&		\\\hline
11	&		&	10	&		&		&		&		\\\hline
3	&		&		&		&		&	9	&		\\\hline
	&		&		&	8	&		&	2	&		\\\hline
\end{array}$}\ ,
\end{center}
\noindent where empty entries of $W$ indicate 0. Clearly, $G(W)=[1,14]$ and there are two non-zero
entries in each row, each column and each main diagonal of $W$. On the other
hand, the set of row-sums $R(W)=\{8,20,18,16,21,12,10\}$, the set of column-sums $C(W)=\{14,13,22,9,17,11,19\}$,
$l(W)=15$ and $r(W)=23$, it follows that
$S_W=R(W)\cup C(W)\cup \{l(W),r(W)\}=[8,23]$. So, $W$ is a regular SAMS$(7,2)$.
\end{proof}

The following example is very similar to the above, so
we only list the arrays $A$, $B$, $W$ and $W^*$ by using the proof of Theorem \ref{SAMS(n,2)}.
\begin{exam}\label{Ex2}
There exists a regular SAMS$(11,2)$.
\end{exam}
\begin{proof}
We have $m=5$, $m+2=7$ and $k=2$.

\begin{center}
$A=\renewcommand\arraystretch{0.8}\left(
        \begin{smallmatrix}
12	&	13	&	14	&	15	&	17	&	18	&	19	&	20	&	21	&	22	&	11	\\
1	&	2	&	3	&	4	&	5	&	16	&	6	&	7	&	8	&	9	&	10	\\
        \end{smallmatrix}
        \right),$\ \ \
  $B=\renewcommand\arraystretch{0.8}\left(
        \begin{smallmatrix}
2	&	3	&	4	&	\textbf{5}	&	6	&	\textbf{7}	&	8&9&10&11&1	\\
4	&	\textbf{5}	&	6	&	\textbf{7}	&	8	&	9	&	10&11&1&2&3	\\
6	&	\textbf{7}	&	8	&	9	&	10	&	11	&	1&2&3&4&\textbf{5}	\\
8	&	9	&	10	&	11	&	1	&	2	&	3&4&\textbf{5}&6&\textbf{7}	\\
10	&	11	&	1	&	2	&	3	&	4	&	\textbf{5}&6&\textbf{7}&8&9	\\
1	&	2	&	3	&	4	&	\textbf{5}	&	6	&	\textbf{7}&8	&	9	&	10	&	11	\\
3	&	4	&	\textbf{5}	&	6	&	\textbf{7}	&	8	&	9	&	10	&	11&1	&	2	\\
\textbf{5}	&	6	&	\textbf{7}	&	8	&	9	&	10	&	11&1	&	2	&	3	&	4	\\
\textbf{7}	&	8	&9	&	10	&	11	&	1&	2	&	3	&	4	&	\textbf{5}	&	6	\\
9	&		10	&11&1&	2	&	3	&	4	&	\textbf{5}&6	&	\textbf{7}	&	8	\\
11	&	1	&	2	&	3	&	4	&	\textbf{5}&6	&	\textbf{7}	&	8	&	9	&	10	\\
        \end{smallmatrix}
        \right),$
\end{center}

\begin{center}
 {\renewcommand\arraystretch{0.8}
\setlength{\arraycolsep}{3pt}
\footnotesize
\hspace{15pt}
$W=\begin{array}{|c|c|c|c|c|c|c|c|c|c|c|}\hline
	&		&		&	11	&		&	1	&		&		&		&		&		\\\hline
	&	22	&		&	10	&		&		&		&		&		&		&		\\\hline
	&	9	&		&		&		&		&		&		&		&		&	21	\\\hline
	&		&		&		&		&		&		&		&	20	&		&	8	\\\hline
	&		&		&		&		&		&	19	&		&	7	&		&		\\\hline
	&		&		&		&	18	&		&	6	&		&		&		&		\\\hline
	&		&	17	&		&	16	&		&		&		&		&		&		\\\hline
15	&		&	5	&		&		&		&		&		&		&		&		\\\hline
4	&		&		&		&		&		&		&		&		&	14	&		\\\hline
	&		&		&		&		&		&		&	13	&		&	3	&		\\\hline
	&		&		&		&		&	12	&		&	2	&		&		&		\\\hline
\end{array}$}\ .
\end{center}
We exchange column $2$ with column $4$ and  column $10$ with column $8$ of $W$ to obtain $W^*$ as follows.
\begin{center}
 {\renewcommand\arraystretch{0.8}
\setlength{\arraycolsep}{3pt}
\footnotesize
\hspace{15pt}
$W^*=\begin{array}{|c|c|c|c|c|c|c|c|c|c|c|}\hline
	&	11	&		&		&		&	1	&		&		&		&		&		\\\hline
	&	10	&		&	22	&		&		&		&		&		&		&		\\\hline
	&		&		&	9	&		&		&		&		&		&		&	21	\\\hline
	&		&		&		&		&		&		&		&	20	&		&	8	\\\hline
	&		&		&		&		&		&	19	&		&	7	&		&		\\\hline
	&		&		&		&	18	&		&	6	&		&		&		&		\\\hline
	&		&	17	&		&	16	&		&		&		&		&		&		\\\hline
15	&		&	5	&		&		&		&		&		&		&		&		\\\hline
4	&		&		&		&		&		&		&	14	&		&		&		\\\hline
	&		&		&		&		&		&		&	3	&		&	13	&		\\\hline
	&		&		&		&		&	12	&		&		&		&	2	&		\\\hline
\end{array}$}\ .
\end{center}
\noindent Here empty entries of $W$ and $W^*$ indicate 0. It is easy to see that $W^*$ is a regular SAMS$(11,2)$.
\end{proof}

\section{Symmetric diagonal Kotzig array and symmetric forward diagonals array }

In this section, we introduce a symmetric diagonal Kotzig array and symmetric forward diagonal array
which are the important building blocks in our construction next section.

\begin{defi}\label{symdiaKA}
Suppose $n$ and $d$ are positive integers with $d \leq n$. A $d\times n$ rectangular array $A=(a_{i,j})$, $i\in I_d$,
$j\in I_n$, is a \emph{symmetric diagonal Kotzig array} if it has the following properties:

1. Each row is a permutation of the set $I_n=\{1,2,\cdots, n\}$.

2. All columns have the same sum.

3. All forward diagonals have the same sum.

4. $a_{i,j}+a_{d+1-i,n+1-j}=n+1$  for each $(i,j)\in I_d\times I_n$.
\end{defi}

Three-row arrays satisfying the first two conditions of the Definition \ref{symdiaKA} were used by A. Kotzig
( \cite{Kotzig}) to construct edge-magic labelings and there is an account of this in \cite{Wallis,Wallis1}
where they are called \emph{Kotzig arrays}. I. Gray and J. MacDougall have constructed a $d$-row generalization of these Kotzig arrays and they have been used to construct vertex-magic labelings for complete bipartite graphs ( \cite{Gray1}).
The arrays satisfying the first three conditions of the  Definition \ref{symdiaKA} were used by I. Gray and J. MacDougall (\cite{Gray}) to construct sparse semi-magic square and vertex-magic labelings, and they are called \emph{diagonal Kotzig arrays}. Our constructions of squares require the diagonal Kotzig arrays with the additional diagonal condition stated
as property 4 above.

\begin{defi}
Suppose $n$ and $t$ are positive integers and $t \leq n$. A $t\times n$ array
$A=(a_{i,j})$, $i\in I_t, j\in I_n$, is a \emph{symmetric forward diagonals array}, denoted by SFD$(t,n)$ for short, if it satisfies
the following properties:

1. The elements set of $A$ consists of $nt$ consecutive positive integers.

2. All columns have the same sum.

3. All forward diagonals have the same sum.

4. $a_{i,j}+a_{t+1-i,n+1-j}$ is a constant for any $(i,j)\in I_t\times I_n$.

\end{defi}

If $A_1=(a_{i,j}^{(1)})$ is an SFD$(t,n)$ over $I_{nt}$, let
$a_{i,j}^{(2)}=a_{i,j}^{(1)}+l$, where $l$ is a nonnegative integer,
then $A_2=(a_{i,j}^{(2)})$ is also an SFD$(t,n)$ over $[1+l,nt+l]$.

\begin{con}\label{SymDK-SFD}
If there exists a symmetric diagonal Kotzig array of order $d\times n$,
 then there exists an SFD$(d,n)$.
\end{con}
\begin{proof}
Let $A=(a_{i,j})$ be  a symmetric diagonal Kotzig array of order $d\times n$ and
$B=(b_{i,j})$ be the $d\times n$ array with $b_{i,j}=i-1$,
where $i\in I_d$, $n\in I_n$. Next we shall show that
$S=A+nB=(s_{i,j})$ is an SFD$(d,n)$.

Clearly, $\bigcup\limits_{i=1}^{d}\bigcup\limits_{j=1}^{n}\{s_{i,j}\}=I_{dn}$.
Note that the columns of $A$ and $B$ have constant sum respectively and therefore the columns
of $S = A + nB$ will also have a constant sum $k$. Also the forward diagonals
of $A$ and $B$ have constant sum respectively and so the forward diagonals of $S$ will also
have constant sum, also equal to $k$. Since $a_{i,j}+a_{d+1-i,n+1-j}$ is a constant, then
$$s_{i,j}+s_{d+1-i,n+1-j}=(a_{i,j}+nb_{i,j})+(a_{d+1-i,n+1-j}+nb_{d+1-i,n+1-j})=a_{i,j}+a_{d+1-i,n+1-j}+n(n+1)$$
is a constant. Hence $S$ is an SFD$(d,n)$.
\end{proof}

So in order to show the existence of a SFD$(d,n)$, we only show how
to construct a symmetric diagonal Kotzig array.
The following theorem is obtained by using direct
construction and the recurrence method.
\begin{thm}\label{SymDKA}
There exists a symmetric diagonal Kotzig array of order $d\times n$ for any odd
integer $n\geq 3$ and integer $d\in [3,n]$.
\end{thm}
\begin{proof}
For $i\in I_3$, $j\in I_n$, let $A_3=(a_{i,j})$,  where

\begin{center}
$    a_{1,j}= \left\{
          \begin{array}{ll}
              n-\frac{j-1}{2},     &  if \ \ \ j\ \  is \ \ odd, \\
               \frac{n+1-j}{2},  & if \ \ \ j \ \   is\ \ even,\\
          \end{array}
       \right. \ \ \
       a_{2,j}= j,\ \ \   a_{3,j}= n+1-a_{1,n+1-j}.$
\end{center}

For $i\in I_4$, $j\in I_n$, let $A_4=(b_{i,j})=\left(
                                               \begin{array}{c}
                                                 B_1 \\
                                              B_2 \\
                                               \end{array}
                                             \right)
$,  where

\begin{center}
$ b_{1,j}= \left\{
          \begin{array}{ll}
              j,     & if \ \ \ j\leq {n-1\over 2}, \\
               j+1,  & if \ \ \ {n+1\over 2 }\leq j\leq n-1,\\
     {n+1\over 2},   & if \ \ \  j=n, \\
          \end{array}
       \right. \ \ \
    b_{2,j}= \left\{
          \begin{array}{ll}
               {n+1\over 2},   & if \ \ \  j=1, \\
               n+2-j ,    & if \ \ \ 2\leq  j\leq {n+1\over 2}, \\
               n+1-j,    & if \ \ \  j> {n+1\over 2},\\
          \end{array}
       \right.$
\end{center}

$$b_{3,j}=n+1-b_{2,n+1-j},\ \ \ \ \ b_{4,j}=n+1-b_{1,n+1-j}.$$
For $i\in I_5$, $j\in I_n$, let $A_5=(c_{i,j})$,  where

$$c_{1,j}=[{n+1\over 2}(j-1)](\emph{mod}\ \ n)+1,\ \ \ c_{2,j}=n+1-j,\ \ \ c_{3,j}=j,$$

$$c_{4,j}=n+1-j,\ \ \ c_{5,j}=j+{n+1\over 2}-c_{1,j}.$$
It is readily checked that $A_3$, $A_4$, $A_5$ and $A_6=\left(\begin{array}{c}
                      A_3 \\
                      A_3 \\
                    \end{array}
                  \right)$ are the symmetric diagonal Kotzig arrays of order $d\times n$ for $d=3,4,5,6$ respectively.

For $d\geq 7$, it can be written as $d=4k+\alpha$, where $\alpha\in\{3,4,5,6\}$. Let
\begin{center}
$E=\left(
        \begin{smallmatrix}
B_1 \\
      \vdots\\
       B_1 \\
       A_{\alpha} \\
       B_2 \\
      \vdots \\
       B_2\\
 \end{smallmatrix}
        \right),$
\end{center}
where $B_i$ occurs $k$ times for $i=1,2$.
It is clear that $E$ is a symmetric diagonal Kotzig array of order $(4k+\alpha)\times n$.
\end{proof}

\noindent\textbf{Remark 2}\ \ \
(i)\ \ It is to be pointed out that the array $A_4$ also has the property that for any $j\in I_n$,
\begin{center}
$b_{1,j}+b_{2,\langle j+1\rangle_n}=n+1$,
\end{center}
so does
\begin{center}
$b_{3,j}+b_{4,\langle j+1\rangle_n}=n+1.$
\end{center}

(ii)\ \ There are many ways to obtain a symmetric diagonal Kotzig array of
order $d\times n$ with $d\geq 7$, here we also give another different combined way below. Let
\begin{center}
$F=\left(
        \begin{smallmatrix}
 A_3\\
 B_1 \\
      \vdots\\
       B_1 \\
       A_{\alpha} \\
       B_2 \\
      \vdots \\
       B_2\\
        A_3\\
 \end{smallmatrix}
        \right),$
\end{center}
where $B_i$ occurs $k-1$ times for $i=1,2$.
Then it is easy to check that $F$ is also a symmetric diagonal Kotzig array of order $(4k+2+\alpha)\times n$.

(iii)\ \ When $d=2e$ and $e\geq 3$, then we can get a
symmetric diagonal Kotzig array of order $d\times n$ by joining two symmetric diagonal Kotzig arrays of order $e\times n$ coming from Theorem \ref{SFD}, which will be used in the  proof of the following conclusions when the
number of the rows of a  symmetric diagonal Kotzig array is even $d\geq 6$.
\qed
\vskip 12pt

Combine Construction \ref{SymDK-SFD} and Theorem \ref{SymDKA}, we have the following theorem.
\begin{thm}\label{SFD}
For any odd $n\geq 3$ and $t\in [3,n]$, there exists an SFD$(t,n)$ over $[1+l,nt+l]$ for any nonnegative integer $l$.
\end{thm}

\noindent\textbf{Remark 3} \ \
By (i) and (iii) of Remark 2, for $e\geq 2$ and any nonnegative integer $l$, there exists an
SFD$(2e,n)$, $F=(f_{i,j})$, over $[1+l,2en+l]$
by using Construction \ref{SymDK-SFD} and Theorem \ref{SymDKA}, and it has an additional properties:
\begin{center}
$f_{i,\langle i+x\rangle_n}+f_{2e+1-i,n+1-\langle i+x\rangle_n}=2en+1+2l$,
\end{center}
and
\begin{center}
$\sum\limits_{i=1}^ef_{i,\langle i+x\rangle_n}+\sum\limits_{i=e+1}^{2e}f_{i,\langle i+x+y\rangle_n}=(2en+1+2l)e$
\end{center}
for any $x,y\in I_n$.

\section{The existence of a regular SAMS$(n,d)$ for $n\equiv1,5 \pmod 6$ and $d\in [6,n-3]$}

In this section, we shall prove that there exists a regular SAMS$(n,d)$ for any $n\equiv1,5 \pmod 6$ and $d\in [6,n-3]$
by using the arrays $B$, $W$ and $W^*$ in the proof of Theorem \ref{SAMS(n,2)} and the existence of an SFD$(d,n)$
from Theorem \ref{SFD} which constructed by Construction \ref{SymDK-SFD} and Theorem \ref{SymDKA}.

To do this, we also introduce a new concept and some very simple and useful results in the following.

\begin{defi}\label{compatibleDY}
Two $m\times n$ arrays $M=(m_{i,j})$ and $N=(n_{i,j})$ are \emph{compatible} if $\Omega(M)\cap\Omega(N)=\emptyset$,
where $\Omega(M)=\{(i,j)|m_{i,j}\neq0,\ i\in I_m, j\in I_n\}$ and $\Omega(N)=\{(i,j)|n_{i,j}\neq0,\ i\in I_m, j\in I_n\}$.
\end{defi}

\begin{lem}\label{compat-R1}
If there exists a regular SMS$(n,d_1)$ and an SAMS$(n,d_2)$, and they are compatible,
then there exists an SAMS$(n,d_1+d_2)$.
\end{lem}
\begin{proof}
Let $M=(m_{i,j})$ be a regular SMS$(n,d_1)$ over $\{0,1,2,\cdots,nd_1\}$, and $N=(n_{i,j})$ be an SAMS$(n,d_2)$ over $\{0,1,2,\cdots,nd_2\}$.
Let $M'=(m_{i,j}')$, where
\begin{center}
$m_{i,j}'= \left\{
          \begin{array}{ll}
             m_{i,j}+nd_2,     &  if \ \ \  m_{i,j}\neq0, \\
              0,  & if \ \ \ m_{i,j}=0.\\
          \end{array}
       \right. $
\end{center}
It is readily checked that $M'+N$ is an SAMS$(n,d_1+d_2)$ over $\{0,1,2,\cdots,n(d_1+d_2)\}$.
\end{proof}


\begin{lem}\label{compat-R2}
If there exists an SMS$(n,d_1)$ and a regular SAMS$(n,d_2)$, and they are compatible,
then there exists an SAMS$(n,d_1+d_2)$.
\end{lem}
\begin{proof}
Let $M=(m_{i,j})$ be an SMS$(n,d_1)$ over $\{0,1,2,\cdots,nd_1\}$, and $N=(n_{i,j})$ be a regular SAMS$(n,d_2)$ over $\{0,1,2,\cdots,nd_2\}$.
Let $N'=(n_{i,j}')$, where
\begin{center}
$n_{i,j}'= \left\{
          \begin{array}{ll}
             n_{i,j}+nd_1,     &  if \ \ \  n_{i,j}\neq0, \\
              0,  & if \ \ \ n_{i,j}=0.\\
          \end{array}
       \right. $
\end{center}
It is readily checked that $M+N'$ is an SAMS$(n,d_1+d_2)$ over $\{0,1,2,\cdots,n(d_1+d_2)\}$.
\end{proof}

\begin{thm}\label{compat-R3}
If there exists a regular SMS$(n,d_1)$ and a regular SAMS$(n,d_2)$, and they are compatible,
then there exists a regular SAMS$(n,d_1+d_2)$.
\end{thm}

\begin{thm}\label{mainth4-1}
There exists a regular SAMS$(n,d)$ for any $n\equiv 1,5 \pmod 6$, $n\geq 11$ and $d\in [6,n-3]$.
\end{thm}
\begin{proof}
Let $d=t+2$, where $t\in [4,n-5]$, and $m=\frac{n-1}{2}$. By Theorem \ref{SAMS(n,2)} there exists a regular SAMS$(n,2)$, $W$ or $W^*$. By Theorem \ref{compat-R3}, to show the conclusion,  we need only to
construct a regular SMS$(n,t)$, which is compatible with the regular SAMS$(n,2)$.

By Theorem \ref{SFD} there exists an SFD$(t,n)$ over $[1,nt]$, denoted by $C=(c_{i,j})$.
The Latin square of order $n$, $B$, and the function $f$ are both from the proof of Theorem \ref{SAMS(n,2)}.
When $t=2e+1$, we put $c_{i,j}$, $i\in I_t$, $j\in I_n$, into entry
$(f(\langle2i-2e-1+m\rangle_n,\langle2j+m\rangle_n),\langle 2j+m\rangle_n)$ of $B$,
the other entries of $B$ are filled by $0$, denoted by $D$.
When $t=2e$, we put $c_{i,j}$,
$i\in I_{t}$, $j\in I_n$, into entry $(f(\langle2i'-2e-1+m\rangle_n,\langle2j+m\rangle_n),\langle 2j+m\rangle_n)$ of $B$,
where
\begin{center}
$ i'= \left\{
          \begin{array}{ll}
             i,     &  if \ \ \ i\in[1,e], \\
               i+1,  & if \ \ \ i\in[e+1,2e].\\
          \end{array}
       \right. $
\end{center}
the other entries of $B$ are filled by $0$, also denoted by $D$.

Firstly, we shall show that $D$ is a regular SMS$(n,t)$.

(i) Note that we put the elements in the $j$-th column
of $C$ into the $\langle 2j+m\rangle_n$-th column of $D$, where $j\in I_n$,
and $\{j|j\in I_n\}=\{\langle 2j+m\rangle_n|j\in I_n\}=I_n$, therefore there are $t$ non-zero
elements in each column of $D$ and  the columns of $D$ will
have a constant sum $\frac{(1+nt)t}{2}$ since the columns of $C$ have constant sum $\frac{(1+nt)t}{2}$.

(ii)
For each $ j\in I_n$, the elements in the set $\mathcal{A}_1=\{c_{i,\langle j+i\rangle_n}| i\in I_t\}$ are putted into the
same row of $D$ and it is clear that $|\mathcal{A}_1|=t$ and the elements in the set $\mathcal{A}_1$ are
exactly in the same froward diagonal of $C$. In fact, the element $c_{i,\langle j+i\rangle_n}$ is putted into
$f(\langle2i-2e-1+m\rangle_n,\langle2\langle j+i\rangle_n+m\rangle_n)$-th row of $D$.
Let
$f(\langle2i-2e-1+m\rangle_n,\langle2\langle j+i\rangle_n+m\rangle_n)=\alpha$, then by the definition of the function $f$ from the proof
of Theorem \ref{SAMS(n,2)}, we have
\begin{center}
$b_{\alpha,\langle2\langle j+i\rangle_n+m\rangle_n}=\langle2i-2e-1+m\rangle_n=\langle2\alpha+2(j+i)+m-1\rangle_n$,
\end{center}
it follows that
$\alpha=\langle-j-e\rangle_n$ and $\{j|j\in I_n\}=\{\langle-j-e\rangle_n|j\in I_n\}=I_n$, which is independent of the parameter $i$.
This implies that the elements in the set $\mathcal{A}_1=\{c_{i,\langle j+i\rangle_n}| i\in I_t\}$ lie in the
same row of $D$. Clearly, there are $t$ non-zero elements in each row of $D$ from $|\mathcal{A}_1|=t$ and
so all forward diagonals of $C$ become the rows of $D$.
Then
the rows of $D$ will also have a constant sum $\frac{(1+nt)t}{2}$ since all forward diagonals of $C$ have the same sum $\frac{(1+nt)t}{2}$.

(iii) Let 
\begin{center}
$ \Delta= \left\{
          \begin{array}{ll}
             \bigcup\limits_{i=1}^t\{2i-2e-1+m\},     &  if \ \ \ t=2e+1, \\
              \bigcup\limits_{i=1}^t\{2i'-2e-1+m\},  & if \ \ \ t=2e.\\
          \end{array}
       \right. $
\end{center}
Clearly $|\Delta|=t$.
It is easy to check that there are exactly $t$ non-zero elements in
each main diagonal of $D$ since $B$ is a diagonal Latin square and  for $i_1,j_1\in I_n$,
\begin{center}
$  \left\{
          \begin{array}{ll}
             d_{i_1,j_1}=0,     &  if \ \ \ b_{i_1,j_1}\not\in \Delta, \\
              d_{i_1,j_1}\neq0,  & if \ \ \ b_{i_1,j_1}\in \Delta.\\
          \end{array}
       \right. $
\end{center}
Now we compute the main diagonal-sum of $D$.

When $t=2e$, for $i\in I_e$, $j\in I_n$, the elements $c_{i,j}$ and $c_{t+1-i,n+1-j}$ are putted into entries
$(f(\langle2i-2e-1+m\rangle_n,\langle2j+m\rangle_n),\langle 2j+m\rangle_n)$
and $(f(\langle2(t+1-i+1)-2e-1+m\rangle_n,\langle2(n+1-j)+m\rangle_n),\langle 2(n+1-j)+m\rangle_n)$
of $B$ respectively. Let $f(\langle2i-2e-1+m\rangle_n,\langle2j+m\rangle_n)=\sigma$.
Then we have
\begin{center}
$b_{\sigma,\langle2j+m\rangle_n}=\langle2\sigma+(\langle2j+m\rangle_n)+1\rangle_n=\langle2i-2e-1+m\rangle_n$,
\end{center}

\begin{center}
$\begin{aligned}  \ \ \ \ b_{n+1-\sigma,n+1-\langle2j+m\rangle_n}&=(n+1)-b_{\sigma,\langle2j+m\rangle_n}=
(n+1)-\langle (2i-2e-1+m)\rangle_n.
\end{aligned}$
\end{center}
It is easy to compute that

\begin{center}
$\langle2(t+1-i+1)-2e-1+m\rangle_n=\langle2(2e+2-i)-2e-1+m\rangle_n=\langle(n+1)-(2i-2e-1+m)\rangle_n$,
\end{center}

\begin{center}
$\langle 2(n+1-j)+m\rangle_n=(n+1)-\langle2j+m\rangle_n$.
\end{center}
Therefore,

\begin{center}
$\begin{aligned} & \ \ \ \ f(\langle2(t+1-i)-2e-1+m\rangle_n,\langle2(n+1-j)+m\rangle_n) \\ &
=f(\langle(n+1)-(2i-2e-1+m)\rangle_n,\langle (n+1)-(2j+m)\rangle_n)\\ &=
(n+1)-\sigma.
\end{aligned}$
\end{center}
It follows that
for $i\in I_e$, $j\in I_n$,  $c_{i,j}$ and $c_{t+1-i,n+1-j}$ are putted into the entries
$(\sigma,\langle 2j+m\rangle_n)$ and $(n+1-\sigma,n+1-\langle 2j+m\rangle_n)$
of $B$ respectively.  It is easy to see that
\begin{center}
$d_{\sigma,\langle 2j+m\rangle_n}+d_{n+1-\sigma,n+1-\langle 2j+m\rangle_n}=c_{i,j}+c_{t+1-i,n+1-j}=1+nt$.
\end{center}
Then the sum of elements in each diagonal of $D$
is also a constant sum $\frac{(1+nt)t}{2}$ since there are exactly $t$ non-zero elements in each diagonals.

When $t=2e+1$,
we have $d_{m+1,m+1}=c_{e+1,m+1}=\frac{1+nt}{2}$ because the Remark 3 and the element $c_{e+1,m+1}$ is putted into the entry
\begin{center}
$(f(\langle2(e+1)-2e-1+m\rangle_n,\langle2(m+1)+m\rangle_n),\langle 2(m+1)+m\rangle_n)=(f(m+1,m+1),m+1)=(m+1,m+1)$
\end{center}
of $D$ followed from $b_{i,j}=\langle2i+j-1\rangle_n=\langle2(m+1)+(m+1)-1\rangle_n=m+1$ when $i=j=m+1$.
In the similar way to the proof of the case $t=2e$, we have that for $i\in I_e$, $j\in I_n$,
the elements $c_{i,j}$ and $c_{t+1-i,n+1-j}$ are also putted into the entries
$(\sigma,\langle 2j+m\rangle_n)$ and $(n+1-\sigma,n+1-\langle 2j+m\rangle_n)$
of $B$ respectively.
It follows that the sum of elements in each diagonal of $D$
is also a constant sum $(1+nt)e+\frac{1+nt}{2}=\frac{(1+nt)t}{2}$.

Secondly,
we shall show that $D$ is compatible with the regular SAMS$(n,2)$ constructed from Theorem \ref{SAMS(n,2)}.
When $t=2e+1$, denote

\mbox{}\hspace{0.25in} $\Omega(D)=\{(i,j)|d_{i,j}\neq0,\ i,j\in I_n\}$

\mbox{}\hspace{0.62in} $=\{(f(\langle2i-2e-1+m\rangle_n,\langle2j+m\rangle_n),\langle 2j+m\rangle_n)| i\in I_t,\ j\in I_n\}$

\mbox{}\hspace{0.62in} $=\{(x,y)|b_{x,y}\in\bigcup\limits_{i=1}^t \{\langle2i-2e-1+m\rangle_n\},\ x,\ y\in I_n\}$

\vskip 6pt
\noindent
When $t=2e$, denote

\vskip 6pt
\mbox{}\hspace{0.16in} $\Omega(D)=\{(i,j)|d_{i,j}\neq0,\ i,j\in I_n\}
=\{(x,y)|b_{x,y}\in\bigcup\limits_{i=1}^t \{\langle2i'-2e-1+m\rangle_n\},\ x,\ y\in I_n\}$.\\
Clearly $e=\lfloor\frac{t}{2}\rfloor\leq m-2$ since $t\leq n-5=2m+1-5=2m-4$ followed from $d=t+2\leq n-3$.
So it is easy to verify that
\begin{center}
$\{m-2,m,m+2,m+4\}\cap \{\langle2i-2e-1+m\rangle_n|i\in I_t\}=\emptyset$ when $t=2e$,\\
$\{m-2,m,m+2,m+4\}\cap \{\langle2i'-2e-1+m\rangle_n|i\in I_t\}=\emptyset$ when $t=2e+1$.
\end{center}
It follows that
\begin{center}
$\{(x,y)|b_{x,y}\in\{m-2,m,m+2,m+4\}\}\cap \Omega(D)=\emptyset$.
\end{center}
Let $W$ and $W^*$ be the same as those of Theorem \ref{SAMS(n,2)},  that is,
$W$ is an SAMS$(n,2)$ for $n\equiv 1\pmod 6$ and $W^*$ is an SAMS$(n,2)$ for $n\equiv 5\pmod 6$.
By Remark 1, we have
$\Omega(W)=\{(i,j)|b_{i,j}\in\{m,m+2\},\ i,j\in I_n\}$
and
$\Omega(W^*)\subset\{(i,j)|b_{i,j}\in\{m-2,m,m+2,m+4\},\ i,j\in I_n\}$.
Then $\Omega(W)\cap \Omega(D)=\emptyset$ and $\Omega(W^*)\cap \Omega(D)=\emptyset$,
it follows that $W$ and $D$ are compatible, and $W^*$ and $D$ are compatible.

So $D+W$ and $D+W^*$ are the regular SAMS$(n,d)$s by Theorem \ref{compat-R3}.
\end{proof}

To illustrate the proof of Theorem \ref{mainth4-1}, we give an example in the following.
\begin{exam}\label{Ex4}
There exists a regular SAMS$(n,d)$ for $(n,d)\in\{(11,8),(13,9)\}$.
\end{exam}
\begin{proof}
For $(n,d)=(11,8)$, then $m=5$, by Theorem \ref{SFD}, we get

$$C=\left(\begin{smallmatrix}
11&5&10&4&9&	3	&8&2&7&1&	6		\\
1	&	2	&	3	&	4	&	5	&	6	&	7	&8&9&10&11\\
6&11&5&10&4&9&3&8&2&7&1  	\\
11&5&10&4&9&	3	&8&2&7&1&	6		\\
1	&	2	&	3	&	4	&	5	&	6	&	7	&8&9&10&11\\
6&11&5&10&4&9&3&8&2&7&1  	\\
        \end{smallmatrix}\right)+11\left(\begin{smallmatrix}
0	&	0	&	0	&	0	&	0	&	0	&	0&	0	&	0	&	0	&	0	\\
1	&	1	&	1	&	1	&	1	&	1	&	1&	1	&	1	&	1	&	1	\\
2  &   2	&	2	&	2	&	2	&	2	&	2&	2	&	2	&	2	&	2	\\
3	&	3	&	3	&	3	&	3	&	3	&	3&	3	&	3	&	3	&	3	\\
4	&	4	&	4	&	4	&	4	&	4	&	4&	4	&	4	&	4	&	4	\\
5  &   5	&	5	&	5	&	5	&	5	&	5&	5	&	5	&	5	&	5	\\
        \end{smallmatrix}\right)
        =\left(\begin{smallmatrix}
11&5&10&4&9&	3	&8&2&7&1&	6		\\
12&	13&	14&15&	16&	17&	18&	19&	20&	21&	22\\
28&	33&	27&	32&	26&	31&	25&	30&	24&	29	&23\\
44&38&	43&	37&	42&	36&	41&	35&	40&	34&	39\\
45&	46&	47&	48&	49&	50&	51&	52&	53&	54&	55\\
61&	66&60&	65&	59&	64&	58&	63&	57&	62&	56\\
        \end{smallmatrix}\right).
$$
\noindent
By the proof of Theorem \ref{mainth4-1}, we obtain
$$C'=C+
\left(\begin{smallmatrix}
22	&	22	&	22	&	22	&	22	&	22	&	22&	22	&	22	&	22	&	22	\\
22	&	22	&	22	&	22	&	22	&	22	&	22&	22	&	22	&	22	&	22	\\
22	&	22	&	22	&	22	&	22	&	22	&	22&	22	&	22	&	22	&	22	\\
22	&	22	&	22	&	22	&	22	&	22	&	22&	22	&	22	&	22	&	22	\\
22	&	22	&	22	&	22	&	22	&	22	&	22&	22	&	22	&	22	&	22	\\
22	&	22	&	22	&	22	&	22	&	22	&	22&	22	&	22	&	22	&	22	\\
        \end{smallmatrix}\right)
        =\left(\begin{smallmatrix}
33&	27&	32&	26&	31&	25&	30&	24&	29&	23&	28\\
34&	35&	36&	37&	38&	39&	40&	41&	42&	43&	44\\
50&	55&	49&	54&	48&	53&	47&	52&	46&	51&	45\\
66&	60&	65&	59&	64&	58&	63&	57&	62&	56&	61\\
67&	68&	69&	70&	71&	72&	73&	74&	75&	76&	77\\
83&	88&	82&	87&	81&	86&	80&	85&	79&	84&	78\\
        \end{smallmatrix}\right).
$$
The arrays $B$ and $D$ are listed below by the proof of Theorem \ref{mainth4-1} and $W^*$ comes from Example \ref{Ex2}.

\begin{center}
 {\renewcommand\arraystretch{0.7}
\setlength{\arraycolsep}{4pt}
\footnotesize
        $ B=\begin{array}{|c|c|c|c|c|c|c|c|c|c|c|}   \hline
\textbf{2}	&	3	&	\textbf{4}	&	5	&	6	&	7	&	\textbf{8}	&	9	&	\textbf{10}	&	\textbf{11}	&	 \textbf{1}	\\\hline
\textbf{4}	&	5	&	6	&	7	&	\textbf{8}	&	9	&	\textbf{10}	&	\textbf{11}	&	\textbf{1}	&	2	&	3	 \\\hline
6	&	7	&	\textbf{8}	&	9	&	\textbf{10}	&	\textbf{11}	&	\textbf{1}	&	\textbf{2}	&	3	&	\textbf{4}	 &	 5	\\\hline
\textbf{8}	&	9	&	\textbf{10}	&	\textbf{11}	&	\textbf{1}	&	\textbf{2}	&	3	&	\textbf{4	}&	5	&	6	 &	 7	\\\hline
\textbf{10}	&	\textbf{11}	&	\textbf{1}	&	\textbf{2}	&	3	&	\textbf{4}	&	5	&	6	&	7	&	\textbf{8}	 &	 9	\\\hline
\textbf{1}	&	\textbf{2}	&	3	&	\textbf{4}	&	5	&	6	&	7	&	\textbf{8}	&	9	&	\textbf{10}	&	 \textbf{11}	\\\hline
3	&	\textbf{4}	&	5	&	6	&	7	&	\textbf{8}	&	9	&	\textbf{10}	&	\textbf{11}	&	\textbf{1}	&	 \textbf{2}	\\\hline
5	&	6	&	7	&	\textbf{8}	&	9	&	\textbf{10}	&	\textbf{11}	&	\textbf{1}	&	\textbf{2}	&	3	&	 \textbf{4}	\\\hline
7	&	\textbf{8}	&	9	&	\textbf{10}	&	\textbf{11}	&	\textbf{1}	&	\textbf{2}	&	3	&	\textbf{4}	&	5	 &	 6	\\\hline
9	&	\textbf{10}	&	\textbf{11}	&	\textbf{1}	&	\textbf{2}	&	3	&	\textbf{4}	&	5	&	6	&	7	&	 \textbf{8}	\\\hline
\textbf{11}	&	\textbf{1}	&	\textbf{2}	&	3	&	\textbf{4}	&	5	&	6	&	7	&	\textbf{8}	&	9	&	 \textbf{10}	\\\hline
\end{array}$ \ , \
$D=\begin{array}{|c|c|c|c|c|c|c|c|c|c|c|}        \hline
42	&		&	51	&		&		&		&	66	&		&	68	&	24	&	82	\\\hline
46	&		&		&		&	61	&		&	67	&	30	&	88	&	41	&		\\\hline
	&		&	56	&		&	77	&	25	&	83	&	40	&		&	52	&		\\\hline
62	&		&	76	&	31	&	78	&	39	&		&	47	&		&		&		\\\hline
75	&	26	&	84	&	38	&		&	53	&		&		&		&	57	&		\\\hline
79	&	37	&		&	48	&		&		&		&	63	&		&	74	&	32	\\\hline
	&	54	&		&		&		&	58	&		&	73	&	27	&	85	&	36	\\\hline
	&		&		&	64	&		&	72	&	33	&	80	&	35	&		&	49	\\\hline
	&	59	&		&	71	&	28	&	86	&	34	&		&	55	&		&		\\\hline
	&	70	&	23	&	81	&	44	&		&	50	&		&		&		&	65	\\\hline
29	&	87	&	43	&		&	45	&		&		&		&	60	&		&	69	\\\hline
\end{array}$}
\end{center}
\begin{center}
 {\renewcommand\arraystretch{0.7}
\setlength{\arraycolsep}{3pt}
\footnotesize
\hspace{15pt}
$W^*=\begin{array}{|c|c|c|c|c|c|c|c|c|c|c|}\hline
	&	11	&		&		&		&	1	&		&		&		&		&		\\\hline
	&	10	&		&	22	&		&		&		&		&		&		&		\\\hline
	&		&		&	9	&		&		&		&		&		&		&	21	\\\hline
	&		&		&		&		&		&		&		&	20	&		&	8	\\\hline
	&		&		&		&		&		&	19	&		&	7	&		&		\\\hline
	&		&		&		&	18	&		&	6	&		&		&		&		\\\hline
	&		&	17	&		&	16	&		&		&		&		&		&		\\\hline
15	&		&	5	&		&		&		&		&		&		&		&		\\\hline
4	&		&		&		&		&		&		&	14	&		&		&		\\\hline
	&		&		&		&		&		&		&	3	&		&	13	&		\\\hline
	&		&		&		&		&	12	&		&		&		&	2	&		\\\hline
\end{array}$}\ .
\end{center}
Here, all of above empty entries indicate 0. It is easy to check that $D+W^*$ is a regular
SAMS$(11,8)$.

For $(n,d)=(13,9)$,
then $m=6$, by Theorem \ref{SFD}, we have
$$C=\left(\begin{smallmatrix}
1&	2&	3&	4&	5&	6&	8&	9&	10&	11&	12&	13&	7\\
7&	13&	12&	11&	10&	9&	8&	6&	5&	4	&3&	2&	1\\
13&	6&	12&	5&	11&	4&	10&	3&	9&	2&	8&	1&	7\\
1&	2&	3&	4&	5&	6&	7&	8&	9&	10&	11&	12&	13\\
7&	13&	6&	12&	5&	11&	4&	10&	3&	9&	2&	8&	1\\
13&	12&	11&	10&	9&	8&	6&	5&	4&	3&	2&	1&	7\\
7&	1&	2&	3&	4&	5&	6&	8&	9&	10&	11&	12&	13\\
        \end{smallmatrix}\right)+13\left(\begin{smallmatrix}
0	&	0	&	0	&	0	&	0	&	0	&	0&	0	&	0	&	0	&	0&	0	&	0	\\
1	&	1	&	1	&	1	&	1	&	1	&	1&	1	&	1	&	1	&	1&	1	&	1	\\
2  &   2	&	2	&	2	&	2	&	2	&	2&	2	&	2	&	2	&	2&	2	&	2	\\
3	&	3	&	3	&	3	&	3	&	3	&	3&	3	&	3	&	3	&	3&	3	&	3	\\
4	&	4	&	4	&	4	&	4	&	4	&	4&	4	&	4	&	4	&	4&	4	&	4	\\
5  &   5	&	5	&	5	&	5	&	5	&	5&	5	&	5	&	5	&	5&	5	&	5	\\
6	&	6	&	6	&	6	&	6	&	6	&	6&	6	&	6	&	6	&	6&	6	&	6	\\
7  &   7	&	7	&	7	&	7	&	7	&	7&	7	&	7	&	7	&	7&	7	&	7	\\
        \end{smallmatrix}\right)
        =\left(\begin{smallmatrix}
1&	2&	3&	4&	5&	6&	8&	9&	10&	11&	12&	13&	7\\
20&	26&	25&	24&	23&	22&	21&	19&	18&	17&	16&	15&	14\\
39&	32&	38&	31&	37&	30&	36&	29&	35&	28&	34&	27&	33\\
40&	41&	42&	43&	44&	45&	46&	47&	48&	49&	50&	51&	52\\
59&	65&	58&	64&	57&	63&	56&	62&	55&	61&	54&	60&	53\\
78&	77&	76&	75&	74&	73&	71&	70&	69&	68&	67&	66&	72\\
85&	79&	80&	81&	82&	83&	84&	86&	87&	88&	89&	90&	91\\
        \end{smallmatrix}\right).
$$
By the proof of Theorem \ref{mainth4-1}, we obtain
$$C'=\left(\begin{smallmatrix}
27&	28&	29&	30&	31&	32&	34&	35&	36&	37&	38&	39&	33\\
46&	52&	51&	50&	49&	48&	47&	45&	44&	43&	42&	41&	40\\
65&	58&	64&	57&	63&	56&	62&	55&	61&	54&	60&	53&	59\\
66&	67&	68&	69&	70&	71&	72&	73&	74&	75&	76&	77&	78\\
85&	91&	84&	90&	83&	89&	82&	88&	81&	87&	80&	86&	79\\
104&	103&	102&	101&	100&	99&	97&	96&	95&	94&	93&	92&	98\\
111&	105&	106&	107&	108&	109&	110&	112&	113&	114&	115&	116&	117\\
        \end{smallmatrix}\right).
$$
The arrays $D$ and $W$ are listed below by the proof of
Theorem \ref{mainth4-1} and Theorem \ref{SAMS(n,2)} respectively.
\begin{center}
 {\renewcommand\arraystretch{0.7}
\setlength{\arraycolsep}{4pt}
\footnotesize
        $ D=\begin{array}{|c|c|c|c|c|c|c|c|c|c|c|c|c|}   \hline
&	42	&	&53&	&	78	&	&85&&		103&&		106&37\\\hline
&	60&	&	77&		&79	&	&104	&	&105	&36	&&	43\\\hline
&	76	&&	86	&	&98&	&	111&	35	&&	44&	&	54\\\hline
&	80	&&	92	&&	117&	34	&	&45	&&	61	&&	75\\\hline
&	93	&&	116&	32&	&	47	&&	55	&&	74	&&	87\\\hline
&	115&	31&	&	48&	&	62	&&	73	&&	81	&&	94\\\hline
30&	&	49&	&	56	&	&	72	&&	88	&&	95	&&	114\\\hline
50&	&	63	&&	71	&	&	82	&&	96	&&	113&	29&	\\\hline
57&	&	70	&&	89	&	&	97&		&112&	28	&&	51	&\\\hline
69&	&	83	&&	99	&	&	110&	27	&&	52	&&	64&	\\\hline
90&	&	100	&&	109&		33&		&46	&&	58	&&	68&	\\\hline
101&&	108&39	&	&	40	&	&65&&	67	&&	84&	\\\hline
107&	38&	&	41	&	&	59&	&	66&	&	91&	&	102&	\\\hline
\end{array}$ \ , }
\end{center}
\begin{center}
 {\renewcommand\arraystretch{0.7}
\setlength{\arraycolsep}{4pt}
\footnotesize
        $W=\begin{array}{|c|c|c|c|c|c|c|c|c|c|c|c|c|}   \hline
&	&	&	&	13	&&	1	&	&	&	&	&	&\\\hline
&	&	26	&&	12	&	&	&	&	&	&	&	&\\\hline
25	&&	11	&	&	&	&	&	&	&	&	&	&\\\hline
10	&	&	&	&	&	&	&	&	&	&	&24	&\\\hline
&	&	&	&	&	&	&	&	&	23	&&	9	&\\\hline
&	&	&	&	&	&	&	22	&&	8	&	&	&\\\hline
&	&	&	&	&	21	&&	7	&	&	&	&	&\\\hline
&	&	&	20	&	&19	&	&	&	&	&	&	&\\\hline
&	18	&&	6	&	&	&	&	&	&	&	&	&\\\hline
&	5	&	&	&	&	&	&	&	&	&	&&	17\\\hline
&	&	&	&	&	&	&	&	&	&	16	&&	4\\\hline
&	&	&	&	&	&	&	&	15&	&	3	&	&\\\hline
&	&	&	&	&	&	14	&&	2	&	&	&	&\\\hline
\end{array}$ \ . }
\end{center}
It is easy to verify that $D+W$ is a regular
SAMS$(13,9)$.
\end{proof}

\section{The Proof of Theorem \ref{mainth} }
In this section, we shall give  the proof of our main Theorem \ref{mainth}. Firstly,
we obtain the existence of a regular SAMS$(n,4)$ for $n\equiv1,5 \pmod 6$ and $n\geq 7$ by direct construction.
\begin{thm}\label{mainth5-1}
There exists a regular SAMS$(n,4)$ for any  $n\equiv1,5 \pmod 6$ and $n\geq 7$.
\end{thm}
\begin{proof}
For any  $n\equiv1,5 \pmod 6$ and $n\geq 7$, it can be written as $n=2m+1$, where $m\geq 3$. We construct a
special array $A=(a_{i,j})$, $i\in I_4$, $j\in I_n$, where

$$    a_{1,j}= \left\{
          \begin{array}{ll}
             2n+\frac{j+1}{2},     &  if \ \  j\ \ is \ \ odd, \\
          2n+m+1+ \frac{j}{2},  &  if \ \  j\ \ is \ \ even, \\
             \end{array}
       \right. $$
       \vskip 12pt
 $\mbox{}\hspace{1.50in} a_{2,j}= n+2-j, $

 $$        a_{3,j}= \left\{
          \begin{array}{ll}
           1 ,  &   if  \ \ j=1, \\
            n+\frac{j+1}{2},    &   if  \ \ j\geq3 \ \  is\ \ odd  ,\\
             n+m+1+\frac{j}{2},    &  if  \ \ j \ \ is \ \ even, \\
                        \end{array}  \right.    $$
        \vskip 8pt
          $ \mbox{}\hspace{1.50in}  a_{4,j}=4n+1-j.$

\noindent It is easy to calculate that
            $$A_1=\bigcup\limits_{j=1}^n {\{a_{1,j}\}}=[2n+1,3n],\ \ \
         A_2=\bigcup\limits_{j=1}^n {\{a_{2,j}\}}=[2,n+1],$$

         $$A_3=\bigcup\limits_{j=1}^n{\{a_{3,j}\}}=[n+2,2n]\cup\{1\}, \ \ \
         A_4=\bigcup\limits_{j=1}^n {\{a_{4,j}\}}=[3n+1,4n].$$
 Then $G(A)=A_1\cup A_2\cup A_3\cup A_4=[1,4n]$. By a simple calculation, the following is obtained.

\mbox{}\hspace{0.35in}$C(A)=\bigcup\limits_{j=1}^n {\{\sum\limits_{i=1}^4a_{i,j}\}}=\{\sum\limits_{i=1}^4a_{i,1}\}\cup\left(\bigcup\limits_{e=1}^m{
\{\sum\limits_{i=1}^4a_{i,2e+1},\sum\limits_{i=1}^4a_{i,2e}\}}\right)$

\mbox{}\hspace{1.60in}  $=\{7n+3\}\cup\left(\bigcup\limits_{e=1}^m\{8n+3-2e,9n+4-2e\}\right)$

\mbox{}\hspace{1.60in}  $=\{7n+3,8n+1,9n+2\}\cup\left(\bigcup\limits_{e=1}^{m-1}\{8n+1-2e,9n+2-2e\}\right).$

\mbox{}\hspace{0.80in} $G_3=\bigcup\limits_{e=1}^{m-1}\{a_{1,2e-1}+a_{2,2e}+a_{3,2e+1}+a_{4,2e+2}\}$

\mbox{}\hspace{1.00in}  $=\bigcup\limits_{e=1}^{m-1}\{(2n+e)+(n+2-2e)+(n+e+1)+(4n+1-2e-2)\}$

\mbox{}\hspace{1.00in}  $=\bigcup\limits_{e=1}^{m-1}\{8n+2-2e\}.$

\mbox{}\hspace{0.80in}  $G_4=\bigcup\limits_{e=1}^{m-1}\{a_{1,2e}+a_{2,2e+1}+a_{3,2e+2}+a_{4,2e+3}\}$

\mbox{}\hspace{1.00in}   $=\bigcup\limits_{e=1}^{m-1}\{(2n+m+1+e)+(n+2-2e-1)+(n+m+1+e+1)+(4n+1-2e-3)\}$

\mbox{}\hspace{1.00in}   $=\bigcup\limits_{e=1}^{m-1}\{9n+1-2e\},$

\mbox{}\hspace{0.80in}  $G_5=\{a_{1,n-2}+a_{2,n-1}+a_{3,n}+a_{4,1}\}$

\mbox{}\hspace{1.00in}  $=\{(2n+m)+(n+2-n+1)+(n+m+1)+(4n+1-1)\}$

 \mbox{}\hspace{1.00in} $=\{8n+3\},$

\mbox{}\hspace{0.80in}  $G_6=\{a_{1,n-1}+a_{2,n}+a_{3,1}+a_{4,2}\}$

\mbox{}\hspace{1.00in}    $=\{(2n+m+1+m)+(n+2-n)+1+(4n+1-2)\}$

\mbox{}\hspace{1.00in} $=\{7n+2\},$

  \mbox{}\hspace{0.80in}       $G_7=\{a_{1,n}+a_{2,1}+a_{3,2}+a_{4,3}\}$

 \mbox{}\hspace{1.00in}        $=\{(2n+m+1)+(n+2-1)+(n+m+1+1)+(4n+1-3)\}$

  \mbox{}\hspace{1.00in}         $=\{9n+1\}.$

\noindent Denote $F(C)=G_3\cup G_4\cup G_5\cup G_6\cup G_7$, then $F(C)$ is the set of forward diagonal-sums.
 Clearly,

\mbox{}\hspace{0.00in} $C(A)\cup F(C)$

\mbox{}\hspace{-0.20in} $=\{7n+3,8n+1,9n+2,8n+3,7n+2,9n+1\}\cup\left(\bigcup\limits_{e=1}^{m-1}\{8n+1-2e,8n+2-2e,9n+1-2e,9n+2-2e\}\right)$

\mbox{}\hspace{-0.20in} $=\{7n+3,8n+1,9n+2,8n+3,7n+2,9n+1\}\cup[7n+4,8n]\cup[8n+4,9n]$

\mbox{}\hspace{-0.20in} $=[7n+2,9n+2]\setminus\{8n+2\}.$

\vskip 6pt
Let $B=(b_{i,j})=(\langle2i+j-1\rangle_n)$ be the Latin square of order $n$ on $I_n$ which comes from the proof of Theorem \ref{SAMS(n,2)}. Define
\begin{center}
$i=g_r(i',j')=<i'-j'-1>_n,\ \ j=g_c(i',j')=<2j'-3>_n,\ i'\in I_4,\ j'\in I_n,$
\end{center}
we put the element $a_{i',j'}$ of $A$ into cell $(i,j)$ of $B$, other cells of $B$ are
filled by $0$, denoted by $D=(d_{i,j})$.
The elements in the same column of $A$ are also in the same column of $D$ since $j=g_c(i',j')=<2j'-3>_n$, and
the elements in the same forward diagonal of $A$ become the same row of $D$ since
\begin{center}
$g_r(1,j')=g_r(2,<j'+1>_n)=g_r(3,<j'+2>_n)=g_r(4,<j'+3>_n)=<-j'>_n,\ j'\in I_n.$
\end{center}
Then
\begin{center}
$G(D)=[1,4n]$ and $R(D)\cup C(D)=C(A)\cup F(C)=[7n+2,9n+2]\setminus\{8n+2\}.$
\end{center}
We also have

\mbox{}\hspace{1.00in}  $b_{i,j}=<2i+j-1>_n=<2(g_r(i',j')+g_c(i',j')-1>_n$

\mbox{}\hspace{1.23in}  $=<2(i'-j'-1)+(2j'-3)-1>_n$

\mbox{}\hspace{1.23in}  $=<2i'-6>_n,\ i'\in I_4$.

\noindent
It follows that the element $a_{i',j'}$ of $A$ is putted into cell $(i,j)$ of $B$ with $b_{i,j}=n-4,n-2,n,2$
respectively. There exist exactly $4$ non-zero elements in each row, each column and the right diagonal
of $B$, so does $D$.
It is easy to calculate that

\mbox{}\hspace{0.65in} $r(D)=d_{2,n-1}+d_{n,1}+d_{n-2,3}+d_{n-4,5}$

\mbox{}\hspace{0.98in} $=a_{4,1}+a_{3,2}+a_{2,3}+a_{1,4}$

\mbox{}\hspace{0.98in} $=(4n+1-1)+(n+m+1+1)+(n+2-3)+(2n+m+1+2)$

\mbox{}\hspace{0.98in} $=9n+3$

\noindent
since $d_{2,n-1}=a_{4,1}$ followed from $g_r(4,1)=<4-1-1>_n=2,\ g_c(4,1)=<2-3>_n=n-1$,

\ \ $d_{n,1}=a_{3,2}$\ \ \ \ followed from $g_r(3,2)=<3-2-1>_n=n,\ g_c(3,2)=<2\cdot2-3>_n=1$,

\ \ $d_{n-2,3}=a_{2,3}$ followed from $g_r(2,3)=<2-3-1>_n=n-2,\ g_c(2,3)=<2\cdot3-3>_n=3$,

\ \ $d_{n-4,5}=a_{1,4}$ followed from $g_r(1,4)=<1-4-1>_n=n-4,\ g_c(1,4)=<2\cdot4-3>_n=5$.

\noindent
Note that there exist exactly $4$ non-zero elements in the left diagonal of $B$ when $n\equiv 1,5 \pmod 6$,
so does $D$.

\textbf{Case 1}\ \  $n\equiv 1\pmod 6$ and $n\geq 7$.

There exist exactly $4$ non-zero elements  $a_{1,1}$, $a_{2,2+4k}$, $a_{3,2+2k}$, $a_{4,2}$ in the left diagonal of $D$
because

\mbox{}\hspace{0.50in} $g_r(1,1)=<1-1-1>_n= g_c(1,1)=<2\cdot1-3>_n=n-1$,

\mbox{}\hspace{0.50in} $g_r(2,2+4k)=<2-(2+4k)-1>_n= g_c(2,2+4k)=<2(2+4k)-3>_n=2k$,

\mbox{}\hspace{0.50in} $g_r(3,2+2k)=<3-(2+2k)-1>_n=g_c(3,2+2k)=<2(2+2k)-3>_n=4k+1$,

\mbox{}\hspace{0.50in} $g_r(4,2)=<4-2-1>_n=g_c(4,2)=<2\cdot2-3>_n=1$.

\noindent So

\mbox{}\hspace{0.60in} $l(D)=a_{1,1}+a_{2,2+4k}+a_{3,2+2k}+a_{4,2}$

\mbox{}\hspace{0.90in} $=(2n+1)+(n+2-2-4k)+(n+m+1+1+k)+(4n+1-2)$

\mbox{}\hspace{0.90in} $=8n+2$.

\noindent
Then $D$ is a regular SAMS$(n,4)$.

\textbf{Case 2}\ \ $n\equiv 5\pmod 6$ and $n\geq 5$.

We have exactly four non-zero elements  $a_{1,1},a_{2,1+2k},a_{3,1+4k},a_{4,2}$ in the left diagonal of $D$
because

\mbox{}\hspace{0.40in} $g_r(1,1)=<1-1-1>_n=g_c(1,1)=<2\cdot1-3>_n=n-1$,

\mbox{}\hspace{0.40in} $g_r(2,1+2k)=<2-(1+2k)-1>_n=g_c(2,1+2k)=<2(1+2k)-3>_n=4k-1$,

\mbox{}\hspace{0.40in} $g_r(3,1+4k)=<3-(1+4k)-1>_n=g_c(3,1+4k)=<2(1+4k)-3>_n=2k$,

\mbox{}\hspace{0.40in} $g_r(4,2)=<4-2-1>_n= g_c(4,2)=<2\cdot2-3>_n=1$.

\noindent So

\mbox{}\hspace{0.60in} $l(D)=a_{1,1}+a_{2,1+2k}+a_{3,1+4k}+a_{4,2}$

\mbox{}\hspace{0.90in} $=(2n+1)+(n+2-1-2k)+(n+2k+1)+(4n+1-2)$

\mbox{}\hspace{0.90in} $=8n+2$.

\noindent
Then $D$ is also a regular SAMS$(n,4)$.
\end{proof}

We restate our main theorem in the following and prove it.

\textbf{Theorem 1.3}\ \ \
For any $n\geq 5$ and $n\equiv1,5 \pmod 6$, there exists a regular SAMS$(n,d)$ if and only if $2\leq d\leq n-1$.

\noindent\textbf{ Proof} \ \ \
It is clear that there does not exist
a regular SAMS$(n,1)$.

For each $n\equiv1,5 \pmod 6$ and $n\geq 5$, there exists a
regular SAMS$(n,2)$ by Theorem \ref{SAMS(n,2)}.

For each $n\equiv1,5 \pmod 6$, $d=3,5$ and $n>d$, there exists a
regular SAMS$(n,d)$ by Lemma \ref{03}.

For each $n\equiv1,5 \pmod 6$ and $n\geq 7$, there exists a
regular SAMS$(n,4)$ by Theorem \ref{mainth5-1}.

For each $n\equiv1,5 \pmod 6$ and $d=n-1,n-2$, there exists a
regular SAMS$(n,d)$ by Lemmas \ref{01}-\ref{02}.

For each $n\equiv1,5 \pmod 6$ and $d\in[6,n-3]$, there exists a
regular SAMS$(n,d)$ by Theorem \ref{mainth4-1}. The proof is completed.

\vskip 10pt
\noindent\textbf{Acknowledgements}\ \ The authors would like to thank Professor Zhu Lie of Suzhou University for his encouragement and many helpful suggestions.

\vskip 12pt

\end{document}